\documentclass[11pt]{amsart}
\usepackage{amssymb,amsmath,latexsym,enumerate, dsfont}
\usepackage{graphicx,verbatim,enumerate,%
ifpdf,mdwlist}

\usepackage{color}
\usepackage{mathabx}
\ifpdf
\usepackage{hyperref}
\else
\usepackage[hypertex]{hyperref}
\fi
\usepackage{graphicx}
\usepackage{cancel}
\usepackage{float}
\renewcommand{\phi}{\varphi}

\newtheorem{theorem}{Theorem}[section]

\newtheorem{lemma}[theorem]{Lemma}
\newtheorem{corollary}[theorem]{Corollary}

\newtheorem{defi}[theorem]{Definition}
\newenvironment{emdef}{\begin{defi} \rm}{ \end{defi}}
\newtheorem{exa}[theorem]{Example}
\newenvironment{example}{\begin{exa} \rm}{ \end{exa}}
\newenvironment{remark}{\begin{rem} \rm}{ \end{rem}}

\newtheorem{rem}[theorem]{Remark}

\DeclareMathOperator{\Id}{Id}

\DeclareMathOperator{\Inf}{\mathbf{Inf}}
\DeclareMathOperator{\Tr}{Tr}
\DeclareMathOperator{\NLow}{\mathbf{non-low}}

\DeclareMathOperator{\Hyperdark}{\mathbf{hdark}}
\DeclareMathOperator{\Hyperhyperdark}{\mathbf{hhdark}}

\DeclareMathOperator{\FinCl}{\mathbf{FinCl}}
\DeclareMathOperator{\InfCl}{\mathbf{InfCl}}

\newcommand{\rel}[1]{\mathrel{#1}}

\newcommand{\QA}[1]{\forall{#1}\,}
\newcommand{\QE}[1]{\exists{#1}\,}

\title[Complexity of word problems]{Classifying word problems of finitely
generated algebras via computable reducibility}

\author[V.~Delle Rose]{Valentino Delle Rose}
\address{National Center for Artificial Intelligence (CENIA), Chile}
\email{\href{mailto:valentino.dellerose@cenia.cl}{valentino.dellerose@cenia.cl}}
\author[L.~San Mauro]{Luca San Mauro}
\address{Institute of Discrete Mathematics and Geometry, Vienna University
of Technology, Austria}
\email{\href{mailto:luca.sanmauro@gmail.com}{luca.sanmauro@gmail.com}}

\author[A.~Sorbi]{Andrea Sorbi}
\address{Dipartimento di Ingegneria Informatica e Scienze Matematiche\\
Universit\`a Degli Studi di Siena\\
I-53100 Siena, Italy}
\email{\href{mailto:andrea.sorbi@unisi.it}{andrea.sorbi@unisi.it}}


\thanks{Delle Rose is funded by the National Center for Artificial
Intelligence CENIA FB210017, Basal ANID. San Mauro's research was
funded in whole or in part by the Austrian Science Fund (FWF) [P 36304-N].
Sorbi is a member of INDAM-GNSAGA, and was partially supported by PRIN
2017 Grant ``Mathematical Logic: models, sets,
computability''.}

\thanks{The authors wish to thank an anonymous referee for many precious
suggestions and comments, and for
having pointed out a series of papers
(in particular~\cite{Kasymov-Khusainov}) which had escaped their attention,
but are highly relevant to this paper.}


\begin{document}

\begin{abstract}

We contribute to a recent research program which aims at revisiting the
study of the complexity of word problems, a major area of research in
combinatorial algebra, through the lens of the theory of computably
enumerable equivalence relations (ceers), which has considerably grown  in
recent times. To pursue our analysis, we rely on the most popular way of
assessing the complexity of ceers, that is via computable reducibility on
equivalence relations, and its corresponding degree structure (the
$\mathrm{c}$-degrees). On the negative side, building on previous work of
Kasymov and Khoussainov, we individuate a collection of
$\mathrm{c}$-degrees of ceers which cannot be realized by the word problem
of any finitely generated algebra of finite type. On the positive side,  we
show that word problems of finitely generated semigroups realize a
collection of $\mathrm{c}$-degrees which embeds rich structures and is
large in several reasonable ways.
\end{abstract}

\maketitle

\section{Introduction}
In recent years, computably enumerable (or, simply, c.e.) equivalence
relations, often called \emph{ceers} after~\cite{Gao-Gerdes}, have been
widely studied. One of the reasons motivating this interest lies in the fact
that ceers arise naturally in combinatorial algebra as word problems of
familiar c.e.~algebraic structures like groups, semigroups, rings, and so on.
By a \emph{c.e.~structure} $A$, we will mean in this paper a nontrivial
algebraic-relational structure for which there exists a
\emph{c.e.~presentation}, i.e.\ a structure $A_{\omega}$ of the same type as
$A$ but having universe $\omega$, possessing uniformly computable operations,
uniformly c.e.\ relations, and a ceer $=_{A}$ which is a congruence on
$A_{\omega}$ such that $A \simeq {A_\omega}_{/=_A}$, i.e.\ $A$ is isomorphic
with the quotient structure obtained by dividing $A_{\omega}$ by $= _{A}$.
The ceer $=_{A}$ is called in this case the \emph{word problem} of $A$ (or,
rather, of its given c.e.\ presentation). Selivanov's survey
paper~\cite{Selivanov} (c.e.\ structures are therein called \emph{positive
structures}) and Khoussainov's survey paper~\cite{Bakh} are excellent
introductions to c.e.~structures.

Word problems appeared in mathematics in 1911, when
Dehn~\cite{dehn1911unendliche} introduced the word problem for finitely
presented  groups, with the goal of addressing the topological issue of
deciding whether two knots are equivalent. Nowadays, much is known about the
complexity of word problems for various algebraic structures. Most notably,
the Novikov-Boone theorem~\cite{novikov1955,boone1959word}---one of the most
spectacular applications of computability theory to general
mathematics---states that the word problem for finitely presented\ groups is
undecidable.

Yet, a basic obstacle towards a full understanding of word problems is that
the computability theoretic machinery commonly employed to measure their
complexity  (e.g., Turing reducibility) is defined for sets, while it is
generally acknowledged that many computational facets of word problems emerge
only if one interprets them as equivalence relations. For example, if $G$ is
a c.e.~group, then it is immediate to see that the equivalence classes of
$=_G$ are uniformly computably isomorphic with each other. It follows that in
a group, the individual word problems (i.e., to decide equality to a word
$w$, as $w$ changes) have all the same complexity. This is not the case for
semigroups: Shepherdson~\cite{shepherdson1965machine} proved that from any
uniformly c.e.~sequence $\{A_i: i \in \omega\}$ of sets, one can construct a
finitely presented\ semigroup $S$ such that the collection of the Turing
degrees of the individual word problems of $S$ (i.e.\ the equivalence classes
of $=_S$) contains all the Turing degrees of the various sets $A_i$.

Hence, to enrich the study of word problems, one shall lift the underlying
computability theoretic analysis from sets to equivalence relations. In this
direction, it is important to mention that starting from
\cite{Gavryushkin-Khoussainov-Stephan}, Khoussainov and other authors have
conducted a systematic investigation of which c.e.\ structures have a word
problem \emph{coinciding} with a fixed ceer: see also
\cite{Fokina-Khoussainov-Semukhin, Gavruskin-Jain-Khoussainov-Stephan}.
Particularly important to this line of
research is an early result of Kasymov and
Khoussainov~\cite{Kasymov-Khusainov} implying that no ceer whose principal
transversal (see Definition~\ref{def:principal-transversal}) is hyperimmune
can be the word problem of any finitely generated algebra of finite type.
Other relevant papers, which investigate how the computability theoretic
properties of a ceer affect the algebraic properties of the structures having
that ceer as word problem, are \cite{Kasymov, Khoussainov-Miasnikov}.

In this paper, we take a slightly different approach, namely we are
interested in investigating which ceers can be \emph{identified} in a broader
sense with word problems of which c.e.~structures, where, rather than mere
coincidence, being identified means in this case to lie in the same
reducibility degree with respect to some reducibility on equivalence
relations which is suitable to measure their relative complexity. Pioneering
attempts at this approach can be found for instance in
\cite{Nies-Sorbi,DelleRose-S-S-1}. The most popular reducibility in this
sense is the one given by next definition.

\begin{emdef}\label{def:reducibility}
Given a pair of equivalence relations $R, S$ on $\omega$, we say that $R$ is
\emph{computably reducible} (or, simply, \emph{$\mathrm{c}$-reducible}) to
$S$ (denoted by $R \le_{\mathrm{c}} S$) if there exists a computable function
$f$ such that
\[
(\forall x,y)[x \rel{R} y \Leftrightarrow f(x) \rel{S} f(y)].
\]
\end{emdef}

By means of this reducibility, we identify two equivalence relations $R, S$
if $R\leq_{\mathrm{c}} S$ and $S \leq_{\mathrm{c}} R$ (denoted by
$R\equiv_{\mathrm{c}} S$). The \emph{$\mathrm{c}$-degree} of $R$ is the
equivalence class of $R$ under the equivalence relation
$\equiv_{\mathrm{c}}$. Given a c.e.\ structure $A$, let us say that a ceer
$R$ is \emph{$\mathrm{c}$-realized} by $A$, if $R$ and $=_A$ have the same
$\mathrm{c}$-degree; let us also say that a $\mathrm{c}$-degree of ceers is
\emph{$\mathrm{c}$-realized} by $A$ if some ceer in the $\mathrm{c}$-degree
is $\mathrm{c}$-realized by $A$ (clearly, this is equivalent to saying that
all the ceers in the $\mathrm{c}$-degree are $\mathrm{c}$-realized by $A$).
Special attention should be given to the problem of finding families of
structures which are \emph{$\mathrm{c}$-complete for the ceers}, namely
families of structures such that every ceer $R$ is $\mathrm{c}$-realized by
some structure lying in the family.

As aforementioned, it is not difficult to see that the groups are not
$\mathrm{c}$-complete for the ceers (observed in \cite{Gao-Gerdes}; see
\cite[Fact~3.6]{DelleRose-S-S-1} for a proof). On the other hand, it has been
shown in \cite{DelleRose-S-S-1} that the semigroups are $\mathrm{c}$-complete
for the ceers. In fact, for every ceer $R$ there exists a c.e.~semigroup $S$
such that the two ceers $R$ and $=_S $ are $\mathrm{c}$-equivalent, in fact
they are isomorphic in the category of equivalence relations: recall that two
equivalence relations $U,V$ on $\omega$ are called \emph{isomorphic} if there
is a reduction $f$ of $U$ to $V$ such that the range of $f$ intersects all
$V$-equivalence classes. This implies that the reduction is invertible, i.e.\
there is a reduction $g$ of $V$ to $U$ such that $f$ ad $g$ invert each other
on the equivalence classes, namely $x\rel{U} g(f(x))$ and $x \rel{V}
f(g(x))$, for every number $x$: this isomorphism relation on equivalence
relations has been considered in several papers, including
\cite{Gavruskin-Jain-Khoussainov-Stephan,Gavryushkin-Khoussainov-Stephan,
joinmeet}; for a justificaton of the name ``isomorphism'' given to it, see
\cite{DelleRose-SanMauro-Sorbi:categories}. However, answering a question
raised by Gao and Gerdes in \cite{Gao-Gerdes}, one can build a ceer
(see~\cite{DelleRose-S-S-1}) which is not $\mathrm{c}$-realized by any
finitely generated semigroup, thus showing that the finitely generated
semigroups are not $\mathrm{c}$-complete for the ceers.

In Section~\ref{sct:negative} (where, on the negative side, we are interested
in describing ceers which cannot be $\mathrm{c}$-realized by finitely
generated algebras of finite type) we show (Theorem~\ref{thm:main-first})
that if $R$ is a ``hyperdark'' ceer (by Definition~\ref{def:various-darkness}
this means that $R$ has infinitely many equivalence classes and all of its
infinite transversals are hyperimmune),
then $R$ cannot be
$\mathrm{c}$-realized by any finitely generated~algebra of finite type. Thus,
Theorem~\ref{thm:main-first} generalizes to $\mathrm{c}$-realizability the
above mentioned result of Kasymov and Khoussainov~\cite{Kasymov-Khusainov}
which, in the terminology of Definition~\ref{def:various-darkness}, states
that no hyperdark ceer can coincide with the word problem of any finitely
generated~algebra of finite type: indeed our proof is a straightforward
sharpening of \cite{Kasymov-Khusainov}, based on the observation (see
Lemma~\ref{lem:main-first}) that, for a ceer, the property that every
infinite transversal is hyperimmune is invariant under
$\mathrm{c}$-equivalence. In Section~\ref{sct:finite-classes} we give
examples of hyperdark ceers: in particular, in Theorem~\ref{thm:all-high} we
exhibit an example of a hyperdark ceer whose equivalence classes are all
finite. On the other hand (Theorem~\ref{thm:finite-nohhd}), we show that
ceers having only finite classes are bound to have an infinite transversal
which is not hyperhyperimmune (hence, they cannot be ``hyperhyperdark'', see
again Definition~\ref{def:various-darkness} below).

On the positive side, in Section~\ref{sct:positive} we investigate the
collection $\mathbf{S}_{f.g.}$ of the $\mathrm{c}$-degrees of ceers which are
$\mathrm{c}$-realized by finitely generated semigroups. In this regard,
Theorem~\ref{thm:main-first} is optimal, since if we drop from hyperimmunity
to immunity then $\mathrm{c}$-unrealizability gets lost, as already shown by
known results in the literature, including the remarkable theorem of
Myasnikov~and~Osin~\cite{Myasnikov-Osin}, proving that in fact there exists
an infinite finitely generated c.e.\ group with a word problem in which all
infinite transversals are immune. Another useful example (which we sketch in
some detail in Example~\ref{example:Hirschfeldt-Khoussainov}) of an infinite
two-generator c.e.~semigroup with a word problem in which all infinite
transversals are immune had been exhibited by Hirschfeldt and
Khoussainov~\cite{Hirschfeldt-Khoussainov}. The rest of
Section~\ref{sct:positive} investigates the subclass of $\mathbf{S}_{f.g.}$
consisting  of the $\mathrm{c}$-degreees of ceers possessing infinite
transversals which are not immune. We prove that this subclass of
$\mathbf{S}_{f.g.}$ is large in several reasonable ways: for instance, it
contains an initial segment of the $\mathrm{c}$-degrees of ceers with
infinitely many equivalence classes, which is order-isomorphic with the tree
$\omega^{<\omega}$ of finite strings of natural numbers, partially ordered by
the prefix relation on strings.

\subsection{Notations and background}\label{ssct:notations}

Our main reference for computability theory is Soare's
textbook~\cite{Soare-Book-new}, to which the reader is referred for all
unexplained notions. Throughout the paper when we talk about ``degrees''
without any further specification, we will mean ``Turing degrees''.

Let $\mathbf{Ceers}$ denote the collection of all ceers. A ceer is
\emph{infinite} if it has infinitely many equivalence classes, it is
\emph{finite} otherwise. Let the symbols $\Inf$, $\mathbf{Fin}$, $\FinCl$ and
$\InfCl$ denote, respectively, the collection of infinite ceers, the
collection of finite ceers, the collection of ceers having only finite
equivalence classes, and the collection of ceers having only infinite
equivalence classes.

\begin{emdef}\label{def:infiniteclasses}
The \emph{cylindrification} of a ceer $R$ is the ceer $R_\infty$ given by
\[
\left\langle i, x \right\rangle \rel{R_{\infty}}
\left\langle j, y \right\rangle \Leftrightarrow i \rel{R} j,
\]
where $\langle \cdot, \cdot \rangle$ denotes the Cantor pairing function.
\end{emdef}
Clearly $R_\infty \in \InfCl$, and $R\equiv_{\mathrm{c}} R_\infty$.

A ceer $R$ is \emph{dark} if $R \in \Inf$ and $\Id_\omega \nleq_{\mathrm{c}}
R$, where $\Id_\omega$ denotes the equality relation on $\omega$. Let the
symbol $\mathbf{dark}$ denote the collection of dark ceers. Dark ceers have
been extensively studied in \cite{joinmeet}. Being dark for a ceer can be
conveniently described using the notion of a transversal for an equivalence
relation.

\begin{emdef}\label{def:transversal}
If $R$ is an equivalence relation on $\omega$, we say that a set $T\subseteq
\omega$ is a \emph{transversal} of $R$ if $x\, \cancel{\rel{R}}\, y$, for every
pair of distinct elements $x,y \in T$.
\end{emdef}

It is now easy to check that a ceer $R\in \Inf$ is dark if and only if it
admits no infinite c.e.~transversal, or equivalently every infinite
transversal of $R$ is \emph{immune}, i.e.\ it does not contain any infinite
c.e.\ set.\footnote{We point out that, in the context of c.e.~structures (for
example in \cite{Hirschfeldt-Khoussainov} and \cite{Khoussainov-Miasnikov}),
the terminology \emph{algorithmically finite} algebra is used to denote a
c.e.~algebra whose word problem is dark.}

Other stronger immunity notions have been widely considered in classical
computability theory, and we briefly recall their definitions. An
\emph{array} of sets of natural numbers is a sequence $(X_{n})_{n \in
\omega}$ of sets of natural numbers. We say that a set $X \subseteq \omega$
\emph{is intersected} by a disjoint array $(X_{n})_{n \in \omega}$
(\emph{disjoint} means that $X_{n}\cap X_{m} =\emptyset$ if $n \ne m$) if,
for all $n$, $X_n \cap X \ne \emptyset$. An infinite set is
\emph{hyperimmune} if it is not intersected by any \emph{strong disjoint
array}, i.e.\ a disjoint array $(F_{n})_{n \in \omega}$ of finite sets,
presented by their canonical indices: hence $F_{n}=D_{f(n)}$ for some
computable function $f$. Similarly, an infinite set is
\emph{hyperhyperimmune} if it is not intersected by any \emph{weak disjoint
array}, i.e.\ a disjoint array $(F_{n})_{n \in \omega}$ of finite sets,
presented by their c.e.~indices: hence $F_{n}=W_{f(n)}$ for some computable
function $f$.

In analogy with the definition of a dark ceer, these stronger immunity
notions suggest accordingly the following definition.

\begin{emdef}\label{def:various-darkness}
A ceer $R$ is \emph{hyperdark} (respectively, \emph{hyperhyperdark}) if $R\in
\Inf$ and all of its infinite transversals are hyperimmune (respectively,
hyperhyperimmune).
\end{emdef}

Let us use the notations $\Hyperdark$ and $\Hyperhyperdark$ to denote,
respectively, the collections of hyperdark ceers and hyperhyperdark ceers.
Clearly $\Hyperhyperdark \subseteq \Hyperdark \subseteq \mathbf{dark}$, as
hyperhyperimmunity implies hyperimmunity, which in turn implies immunity.
Counterexamples witnessing proper inclusions among these classes of ceers can
be found by taking suitable \emph{unidimensional} ceers $R_{X}$, i.e.\ ceers
of the form $x \mathrel{R_X} y$ if and only if $x,y \in X$ or $x=y$, where
$X$ is a given c.e.\ set, and recalling some well know facts of classical
computability theory, which allow to draw the following conclusions: if $X$
is simple but not hypersimple then $R_X \in \mathbf{dark} \smallsetminus
\Hyperdark$, and if $X$ is hypersimple but not hyperhypersimple then $R_X \in
\Hyperdark \smallsetminus \Hyperhyperdark$. Obviously, $R_X \in
\Hyperhyperdark$ if $X$ is hyperhypersimple, hence all these classes are
nonempty.

\begin{remark}
In order to distinguish between ceers and their $\mathrm{c}$-degrees, given a
class $\mathbf{P}$ of ceers we shall adopt the convention of denoting by
$\mathbf{P}_{\mathrm{c}}$ the collection of $\mathrm{c}$-degrees of the
members of $\mathbf{P}$.
\end{remark}

\subsection{On the transversals of a ceer}

We conclude this section with some easy but useful observations about the
transversals of ceers, in particular of hyperdark ceers. If $R$ is an
equivalence relation on $\omega$, let us denote
\[
\Tr(R):=\{T \in 2^{\omega}: \textrm{$T$ is a transversal of $R$}\}.
\]
The following definition points out an important element of $\Tr(R)$.
\begin{emdef}\label{def:principal-transversal}
Given an equivalence relation $R$, its \emph{principal transversal} $T_R$ is
the set comprised of the least elements of all $R$-equivalence classes.
\end{emdef}
It is immediate to see that if $R$ is a ceer then its principal transversal
is co-c.e..

The relevance of the principal transversal in the investigation of
hyperdarkness is highlighted by the following observations, where given any
infinite set $A$ of numbers, we denote by $p_A$ the \emph{principal function
of $A$}, i.e.\ the function which enumerates $A$ in order of magnitude.

\begin{lemma}[Folklore]\label{lem:principal-all}
If $R$ is an equivalence relation on $\omega$ with infinitely many
equivalence classes, then for every infinite transversal $T$ of $R$, the
principal function $p_T$ of $T$ \emph{majorizes} the principal function
$p_{T_R}$ of the principal transversal, i.e. $p_T(i) \ge p_{T_R}(i)$ for
every $i \in \omega$.
\end{lemma}

\begin{proof}
Let $T$ be an infinite transversal of $R$ and for simplicity write
$p_{T_R}(i):=m_i$ and $p_{T}(i):=n_i$. Now, either for every $j<i$ there
exists $k<i$ such that $n_j \mathrel{R} m_k$, but then $m_i \leq n_i$ since
$n_i \notin \bigcup_{k<i} [m_k]_R$; or there exists $j<i$ such that  $n_j
\mathrel{\cancel{R}} m_k$ for every $k<i$, but then $n_j \notin \bigcup_{k<i}
[m_k]_R$ hence $m_i \leq n_j< n_i$.
\end{proof}

\begin{corollary}\label{cor:principal-all}
If $R \in \Inf$ then $R\in \Hyperdark$ if and only if $T_R$ is hyperimmune.
\end{corollary}

\begin{proof}
A well-known theorem by Kuznecov, Medvedev, and Uspenskii
(see~\cite[Theorem~5.3.3]{Soare-Book-new}) states that an infinite set $A$ is
not hyperimmune if and only if there exists a computable function majorizing
the principal function of $A$. Therefore if $T_R$ is hyperimmune then so is
any infinite transversal of $R$.
\end{proof}

\section{Ceers not $\mathrm{c}$-realized by finitely generated algebras of
finite
type}\label{sct:negative}

Theorem~\ref{thm:main-first}, the main result of this section, is essentially
a consequence of Theorem~\ref{thm:Kasymov-Bakh} in \cite{Kasymov-Khusainov}
(see also
\cite{Gavruskin-Jain-Khoussainov-Stephan,Gavryushkin-Khoussainov-Stephan}),
with the addition of our Lemma~\ref{lem:main-first}.

\begin{theorem}\label{thm:Kasymov-Bakh}\cite{Kasymov-Khusainov}
If $A=\left( A, F \right)$ is an infinite c.e.\ algebra of finite type (i.e.,
$F$ is a finite set of operations) and the word problem $=_A$ is hyperdark,
then every finitely generated subalgebra of $A$ is finite.
\end{theorem}

\begin{proof}
For later reference we sketch the proof, taken from
\cite{Gavryushkin-Khoussainov-Stephan}. Let $A$ be as in the statement of the
theorem, and for every $f\in F $ let $n_f$ denote the arity of $f$. Suppose
that $X\subseteq A$ is finite, with $X \ne \emptyset$, but the subalgebra
$A_X$ of $A$, generated by $X$, is infinite. Define the sequence $(X_i)_{i
\in \omega}$ of sets as follows. Let $X_0:=X$; having defined $X_i$ let
\[
X_{i+1}:=X_i\cup \{y: (\exists f \in F)(\exists \vec{x}\in X_i^{n_f})
[y=f(\vec{x})]\}.
\]
Clearly each $X_i$ is a finite set of which one can uniformly compute the
canonical index, and the union $\bigcup_{i\in \omega} X_i$ gives the universe
of $A_X$. Since $A_X$ is infinite, we have that for every $i$ there exists $y
\in X_{i+1}$ such that $y \mathrel{{\ne}_A} z$ for every $z \in X_i$. Thus we
can define a sequence $(y_i)_{i \in \omega}$ such that $y_i \in X_{i+1}$ and
$y_i \mathrel{\ne_A} y_j$ if $i\ne j$, yielding that the set $T=\{y_i: i \in
\omega\}$ is a transversal of $=_A$. The function $m(i):=\max(X_{i+1})$ is
obviously computable and $\max (\{y_i: i \leq n\}) \leq m(n)$, for every $n$.
On the other hand, it is clear that $p_T(n)\leq \max (\{y_i: i \leq n\})$,
for every $n$. Therefore, $=_A$ is not hyperdark, as its infinite transversal
$T$ is not hyperimmune.
\end{proof}

\begin{corollary}\label{corollary:main-cor}
If $A$ is an infinite finitely generated c.e.\ algebra of finite type then the
word problem $=_A$ of $A$ is not hyperdark.
\end{corollary}

\begin{proof}
Immediate.
\end{proof}

Before proving Theorem~\ref{thm:main-first}, we observe:

\begin{lemma}\label{lem:main-first}
If $R \in \Hyperdark$, $E\in \Inf$, and $E \leq_{\mathrm{c}} R$, then $E \in
\Hyperdark$.
\end{lemma}

\begin{proof}
Suppose that $R \in \Hyperdark$, $E \leq_{\mathrm{c}} R$, and $E \in \Inf
\smallsetminus \Hyperdark$. As $E\equiv_{\mathrm{c}} E_{\infty}$ and $R
\equiv_{\mathrm{c}} R_{\infty}$, we have that $E _\infty \leq_{\mathrm{c}}
R_\infty$. It is easy to see that if $U,V$ are ceers with $U\leq_{\mathrm{c}}
V$ and $V \in \InfCl$ then $U \leq_{\mathrm{c}} V$ via a $1$-$1$ computable
function, see for instance \cite[Remark~1.2]{Andrews-Badaev-Sorbi}. Thus,
suppose that $f_0, f_1$ are $1$-$1$ computabe functions reducing
$E\leq_{\mathrm{c}} E_\infty$ and $E_\infty \leq_{\mathrm{c}} R_\infty$,
respectively. Let $T$ be an infinite non-hyperimmune transversal of $E$.
Clearly the set $\widehat{T}:=(f_1\circ f_0)[T]$ (i.e.\ the image of $T$
under the composition $f_1\circ f_0$) is an infinite transversal of
$R_\infty$, and one easily sees  that $\widehat{T}$ is not hyperimmune,
since, by injectivity, $f_1\circ f_0$ maps any strong disjoint array
intersecting $T$ to a strong disjoint array intersecting $\widehat{T}$. By
Lemma~\ref{lem:principal-all}, it follows that the principal transversal
$T_{R_\infty}$ of $R_{\infty}$ is not hyperimmune, and thus the principal
function $p_{T_{R_\infty}}$ of this transversal is majorized by some
computable function $g$. On the other hand by definition of cylindrification,
for every $i$ we have that $p_{T_{R_\infty}}(i)=\langle n_i, 0\rangle$ for
some $n_i$, and the set $\{n_i: i \in \omega\}$ coincides with the principal
transversal $T_R$ of $R$, with principal function $p_{T_R}(i)=n_i$. It
immediately follows by Corollary~\ref{cor:principal-all} that $T_R$ is not
hyperimmune, as $p_{T_R}(i)=n_i\leq \langle n_i, 0\rangle=
p_{T_{R_\infty}}(i) \leq g(i)$.
\end{proof}

\begin{theorem}\label{thm:main-first} If $R\in \Hyperdark$ then $R$ is not
$\mathrm{c}$-realized by any finitely generated algebra of finite type.
\end{theorem}

\begin{proof}
The claim follows by Corollary~\ref{corollary:main-cor}, and the fact that,
by Lemma~\ref{lem:main-first}, membership in $\Hyperdark$ is
$\equiv_{\mathrm{c}}$-invariant, i.e.\ if $E,R$ are ceers with
$R\equiv_{\mathrm{c}} E$ then $E \in \Hyperdark$ if and only if $R \in
\Hyperdark$.
\end{proof}

\subsection{$\Pi^0_1$~classes consisting of infinite transversals}
An easy consequence of Theorem~\ref{thm:main-first} is that for every
infinite finitely generated c.e.~algebra $A$ of finite type there exists a
nonempty $\Pi^0_1$ class containing only infinite transversals of the word
problem of $A$. Recall that a subset $\mathcal{A}$ of the Cantor space
$2^{\omega}$ is called a \emph{$\Pi^{0}_{1}$~class} if $\mathcal{A}$ has a
\emph{$\Pi^{0}_{1}$ definition}, i.e.\ is of the form $\mathcal{A}=\{A\in
2^{\omega}: (\QA{n})R(A,n)\}$, for some decidable predicate $R\subseteq
2^{\omega}\times \omega$. (The $\Pi^{0}_{1}$~classes are also known as the
\emph{effectively closed} subsets of the Cantor space; it is well-known that
a class $\mathcal{A} \subseteq 2^{\omega}$ is a $\Pi^{0}_{1}$~class if and
only if $\mathcal{A}$ coincides with the collection of the infinite paths of
some decidable tree).

\begin{lemma}\label{lem:transv-pi1}
For every ceer $R$, $\Tr(R)$ is a nonempty $\Pi^{0}_{1}$~class of the Cantor
space.
\end{lemma}

\begin{proof}
The claim follows from the observation that if $R$ is a ceer then
\[
\Tr(R)=\{T\in 2^\omega: (\forall\, x,y)[x,y \in T \; \& \; x
\ne y \Rightarrow x \, \mathrel{\cancel{R}} \, y]\},
\]
which provides a description of $\Tr(R)$ as a $\Pi^{0}_{1}$ set since the
complement of $R$ is co-c.e.\,.
\end{proof}

\begin{lemma}\label{lem:transversals-eff-closed}
If $A$ is an infinite finitely generated c.e.~algebra of finite type, then
$\Tr(=_A)$ contains a nonempty $\Pi^0_1$ class of the Cantor space,
consisting of infinite non-hyperimmune transversals.
\end{lemma}

\begin{proof}
Let $A$ be as in the statement of the lemma and let $\{X_i: i \in \omega\}$
be the class of finite sets constructed in the proof of
Theorem~\ref{thm:Kasymov-Bakh}, starting with $X_0:=X$, a finite set of
generators of $A$. Consider
\[
\mathcal{A}:=\Tr(=_{A})\cap \{T \in 2^{\omega}: (\QA{i>0})[T\cap X_i\ne
\emptyset]\}.
\]
By its very definition, all members of $\mathcal{A}$ are infinite and
non-hyperimmune, and $\mathcal{A}$ is nonempty because it contains the
transversal $T$ built in the proof of Theorem~\ref{thm:Kasymov-Bakh}.
\end{proof}

\begin{remark}
By well-known basis theorems for  $\Pi^{0}_1$ classes of the Cantor space
(see e.g.~\cite{Jockusch-Soare:degrees-of-members}) we have that the class
$\mathcal{A}$ in the proof of Lemma~\ref{lem:transversals-eff-closed} always
contains transversals of special computability-theoretic interest, for
instance transversals of low Turing degree, and transversals of
hyperimmune-free degree (we recall that a set $X \leq_{\mathrm{T}}
\emptyset'$ is \emph{low} if $X' \equiv_{\mathrm{T}} \emptyset'$, and a set
is of \emph{hyperimmune-free degree} if its Turing degree does not contain
any hyperimmune set).

In particular we see that the class $\mathcal{A}$ in the proof of
Lemma~\ref{lem:transversals-eff-closed} contains transversals of
hyperimmune-free degree. Of course, none of these transversals can be the
principal transversal if $=_A$ is undecidable, since for every undecidable
ceer $R\in \Inf$, we have that its principal transversal $T_R$, being
co-c.e.\ and not decidable, has hyperimmune degree.
\end{remark}

\subsection{Hyperdark ceers: some examples}\label{sct:finite-classes}

By the discussion immediately following Definition~\ref{def:various-darkness}
we know that hyperdark ceers do exist, as if $X$ is a hypersimple set then
$R_X$ is hyperdark. This gives an example of a hyperdark ceer $R \notin
\FinCl \cup \InfCl$. Its cylindrification $R_\infty$ provides an example of a
hyperdark ceer $R_\infty \in \InfCl$. We now provide an example lying in
$\FinCl$. We first prove the following lemma.

\begin{lemma}\label{lem:all-high}
There exists a ceer $R\in \FinCl$ such that $\emptyset ' \le_{\mathrm{T}} T$,
for every infinite transversal $T$ of $R$.
\end{lemma}

\begin{proof}
The proof will make use of the well-known result (proved by
Martin~\cite{Martin1963}, and independently by Tennenbaum~\cite{Tennenbaum})
that a sufficient condition for $\emptyset' \le_{\mathrm{T}} A$ is the
existence of a function $g \le_{\mathrm{T}} A$ which \emph{dominates} every
partial computable function, i.e.\ for every $e$ there exists a number $i_e$
such that for every $i\ge i_e$, if $\phi_{e}(i)\downarrow$ then $\phi_{e}(i)<
g(i)$. Hence, to complete our task, it will be enough to build a ceer $R$
with only finite equivalence classes and such that the function $n \mapsto
p_{T_R}(n+1)$ dominates every partial computable function (where we recall
that $T_{R}$ denotes the principal transversal of $R$, and, given an infinite
set $A$, the symbol $p_A$ denotes the principal function of $A$). This will
show, as argued at the end of the proof, that for every infinite transversal
$T$ of $R$, the function $g(n):=p_{T}(n+1)$ dominates all partial computable
functions, and clearly $g\leq_{\mathrm{T}} T$.
	
\smallskip
\emph{Construction.} Without loss of generality, we assume that for every
$e,i,s$, if $\phi_{e,s}(i)\downarrow$ then $\phi_{e,s}(i)< s$. For every $e,
s$, let
$$
f_s(e):=  \max \left( \{0\} \cup \{ y: (\QE{i,j \le e})[\varphi_{i,s}(j)
\downarrow=y]\}\right).
$$
For all $e, s$, it holds that $f_s(e) \le f_s(e+1)$ and $f_{s}(e) \le
f_{s+1}(e)$. Moreover, for every $e$ there is a stage $u$ such that, for
every $s \ge u$, $f_s(e) = f_{u}(e)$. Hence, $f(e)= \lim\limits_{s
\rightarrow \infty} f_s(e)$ is well-defined for every $e$. To achieve our
goal, we will try to satisfy, for every $e$, the requirement
$$
\mathcal{R}_e: f(e) < p_{T_R}(e+1),
$$
while guaranteeing that each $R$-equivalence class is finite: this latter
goal will be achieved by building $R$ as a ceer yielding a partition of
$\omega$ in consecutive closed finite intervals. Notice that if $f(e) <
p_{T_R}(e+1)$ for every $e$, then $\phi_{e}(i)< p_{T_R}(i+1)$, for every pair
of numbers $e,i$ such that $e \le i$, and $\phi_{e}(i)\downarrow$.

For the requirements, consider the priority ordering $\mathcal{R}_i <
\mathcal{R}_j$, if $i < j$.

We define $R$ in stages, building a uniformly computable sequence
$\{R_s\}_{s\in \omega}$ of decidable ceers, such that $R_s\subseteq R_{s+1}$
and $R=\bigcup_{s\in \omega} R_s$. At each stage $s$, our approximation
$R_{s}$ to $R$ will be an equivalence relation partitioning $\omega$ in
consecutive closed finite intervals $\{I_{j,s}: j \in \omega\}$, in such a
way that the $R_s$-equivalence of any $x\geq s$ is a singleton.

\emph{Stage $0$}. Start up with $I_{j,0}:=\{j\}$, for every $j$.
Consequently, $R_{0}=\Id_\omega$.

\emph{Stage $s+1$}. We say that a requirement $\mathcal{R}_e$ \emph{requires
attention at stage $s+1$} if $f_{s+1}(e)>\max (I_{e,s})$. By our assumption
on how to approximate the partial computable functions, we may suppose that
$f_{s+1}(e)<s+1$. So, at stage $s+1$, see if there is a requirement
$\mathcal{R}_e$ with $e\leq s$ which requires attention. If not, then go to
stage $s+2$, leaving unchanged each $I_{j}$. Otherwise, let $\mathcal{R}_e$
be the highest priority requirement which requires attention. Define
\[
I_{j,s+1}
:=
\begin{cases}
I_{j,s}, &\textrm{if $j<e$},\\
\left[ \min (I_{e,s}), s \right], &\textrm{if $j=e$},\\
\{s+j-e\}, & \textrm{if $j>e$}.
\end{cases}
\]
We say in this case that $\mathcal{R}_e$ \emph{acts}; clearly, we have that
$f_{s+1}(e) \leq \max(I_{e,s+1})$ after acting. Notice that every $I_{j,s+1}$
is obtained by collapsing finitely many consecutive intervals into just one
interval; clearly $\max (I_{j,s}) \leq \max (I_{j,s+1})$, for every $j$. The
ceer $R_{s+1}$ is the ceer corresponding to the new family of intervals
$\{I_{j,s+1}: j\in \omega\}$; clearly $R_{s} \subseteq R_{s+1}$; notice also
that the $R_{s+1}$-equivalence of any $x\geq s+1$ is a singleton. Go to the
next stage.

\smallskip
\emph{Verification.} A straightforward argument by induction on the priority
of the requirements shows that  for every $e$ the requirement
$\mathcal{R}_{e}$ eventually stops requiring attention, the set $I_{e}$
reaches its limit, and $f(e)\leq \max (I_e)$. To see these claims, let $s_0$
be the least stage such that for all $i < e$, we have that $\mathcal{R}_{i}$
does not receive attention and $I_{i}$ does not change at any stage $s>s_0$.
Notice that $\mathcal{R}_{e}$ may require attention at most finitely many
times after $s_{0}$, as $f_{s}(e)$ may change only finitely many times, and
if $\mathcal{R}_e$ acts at a stage $u$ such that $f_{u}(e)=f(e)$, then it
will never require attention again at any later stage $t$ since $f(e)\leq
\max (I_{e,u}) \leq \max (I_{e,t})$ for every $t \ge u$. As a consequence,
either $\mathcal{R}_e$ never requires attention at any $s \ge s_0$, and thus
$I_{e,s}=I_{e, s_0}$ for every $s\ge s_0$,\; or $\mathcal{R}_e$ requires
attention for the last time at some $s_1\ge s_0$, giving that $I_{e,s}=I_{e,
s_1}$ for every $s \ge s_1$. In either case $I_{e,s}$ reaches its limit
$I_e$, and $f(e) \leq \max (I_e)$.

Having shown that for every $e$, the interval $I_e$ reaches its limit and
$f(e) \leq \max (I_e)$, it is now straightforward to conclude that $f(e) <
\min (I_{e+1})=p_{T_{R}}(e+1)$. By Lemma~\ref{lem:principal-all}, if $T$ is
any infinite transversal of $R$, then $p_{T_{R}}(e+1)\leq p_T(e+1)$, and thus
the function $g(e):=p_T(e+1)$ dominates all partial computable functions.
\end{proof}

\begin{theorem}\label{thm:all-high}
$\Hyperdark \cap \FinCl \ne \emptyset$.
\end{theorem}

\begin{proof}
We first show that if $R\in \Inf$ is a ceer such that $\emptyset'
\le_{\mathrm{T}} T$ for every infinite transversal $T$ of $R$, then $R \in
\Hyperdark$. Clearly, for such an $R$ we have $R \in \NLow$, where $\NLow$
denotes the class of infinite ceers possessing no low infinite transversal.
On the other hand, one can show that $\NLow \subseteq \Hyperdark$. To see
this, assume that $U$ is an infinite ceer such that $U \notin \Hyperdark$,
thus $U$ possesses an infinite transversal $T$ which is not hyperimmune as
witnessed by a strong disjoint array $(D_{f(n)})_{n \in \omega}$, and
consider the class of sets
\[
\mathcal{A}:=\Tr(U)\cap \{X \in 2^\omega: (\QA{n})[X \cap D_{f(n)} \ne \emptyset ]\}.
\]
It is easy to see that $\mathcal{A}$ is a $\Pi^{0}_{1}$~class of the Cantor
space, as by Lemma~\ref{lem:transv-pi1} $\mathcal{A}$ is the intersection of
two $\Pi^{0}_{1}$~classes. Moreover, by the very definition of $\mathcal{A}$,
it is clear that all members of $\mathcal{A}$ must be infinite, and
$\mathcal{A}\ne \emptyset$ since $T \in \mathcal{A}$. Therefore, by the Low
Basis Theorem for $\Pi^0_1$ classes (see e.g.~\cite{Soare-Book-new}),
$\mathcal{A}$ contains a low member, that is an infinite low transversal of
$U$, thus $U \notin \NLow$. By contrapositive, this shows that $\NLow
\subseteq \Hyperdark$.

The theorem now follows from Lemma~\ref{lem:all-high}.
\end{proof}

Theorem~\ref{thm:all-high} exhibits a ceer $R\in \FinCl$ such that every
infinite transversal of $R$ is of hyperhyperimmune degree. In fact, our
example is built so that every infinite transversal of $R$ computes
$\emptyset'$, and thus it is of hyperhyperimmune degree by
\cite[Corollary~4.2]{Jockusch-hhi}. On the other hand at least one of the
infinite transversals of $R$ is not hyperhyperimmune, because our next
theorem shows that hyperhyperdarkness becomes an empty notion, when
considering only ceers in $\FinCl$.

\begin{theorem}\label{thm:finite-nohhd}
$\Hyperhyperdark \cap  \FinCl =\emptyset$.
\end{theorem}

\begin{proof}
Suppose that $R \in \FinCl$. We are going to show that there exists a
transversal $T \in \Tr(R)$ and a weak disjoint array $(F_n)_{n \in \omega}$
which intersects $T$, so that $T$ is not hyperhyperimmune. Throughout the
proof we refer to some fixed computable approximation $\{R_s: s \in \omega\}$
to $R$, namely a sequence of uniformly decidable equivalence relations
$\{R_s\}_{s\in \omega}$, such that $R_0 := \Id_\omega$, $R_s \subseteq
R_{s+1}$, and $R := \bigcup_{s \in \omega} R_s$: every ceer has such an
approximation, see e.g.\ \cite[Lemma~1.4]{Andrews-Badaev-Sorbi}. We construct
$(F_n)_{n \in \omega}$ in stages, so that at stage $s$ we define $F_{n,s}$
for every $n$, and $(F_{n,s})_{n, s \in \omega}$ is a strong array which
provides a uniformly computable approximation to the desired weak array
$(F_n)_{n \in \omega}$, where $F_n:=\bigcup_{s\in \omega} F_{n,s}$.

At stage $0$ let $F_{n,0}:=\emptyset$ for every $n$. At stage $s+1$, let $n$
be the least number such that $F_{n,s} \subseteq \bigcup_{i<n}
[F_{i,s}]_{R_s}$ (such an $n$ exists since all but finitely many $F_{m,s}$
are empty). Pick the least fresh number $x$ (thus $x$ does not lie in any
$F_{m,s}$) and define $F_{n,s+1}:=F_{n,s}\cup \{x\}$; for all $m\ne n$ let
$F_{m,s+1}:=F_{m,s}$.

\smallskip
\emph{Verification.} The array $(F_n)_{n \in \omega}$ is clearly disjoint. An
easy inductive argument shows: for every $n$ there exists a least stage $s_n$
such that $F_{n,s}=F_{n,s_n}$ for all $s > s_n$ (we use here that $R\in
\FinCl$); for all $n$, $F_n \smallsetminus \bigcup_{m<n} [F_m]_R \ne
\emptyset$. Therefore $(F_n)_{n \in \omega}$ is a weak disjoint array, such
that for every $n$, the set $T_n:=F_n \smallsetminus \bigcup_{m<n}[F_m]_R$ is
nonempty. For every $n$, pick $t_n:=\min (T_n)$. Then $T:=\{t_n: n \in
\omega\}$ is an infinite transversal of $R$, which is not hyperimmune since
it is intersected by the disjoint weak array $(F_n)_{n \in \omega}$.
\end{proof}

\section{Ceers $\mathrm{c}$-realized by finitely generated semigroups}
\label{sct:positive}

In the previous section we have isolated a class of infinite ceers (the
hyperdark ceers) which are not $\mathrm{c}$-realized by any finitely
generated algebra of finite type. If we restrict our attention to finitely
generated semigroups, and denote by $\mathbf{S}_{f.g.}$ the structure of
$\mathrm{c}$-degrees of ceers which are $\mathrm{c}$-realized by word
problems of finitely generated\ semigroups, it follows by the results of the
previous section that $\mathbf{S}_{f.g.} \cap \Hyperdark_{\mathrm{c}}
=\emptyset$. In the present section we will face the opposite problem, trying
to individuate classes of $\mathrm{c}$-degrees of ceers which lie in
$\mathbf{S}_{f.g.}$. It is clear that every finite $\mathrm{c}$-degree is in
$\mathbf{S}_{f.g.}$ since every finite ceer is $\mathrm{c}$-realized by some
finite finitely presented group. Therefore, we will confine ourselves to
discuss how large the class $\mathbf{S}_{f.g.}^\infty:= \mathbf{S}_{f.g.}
\cap \Inf_{\mathrm{c}}$ is, i.e.\ the class consisting of the
$\mathrm{c}$-degrees in $\mathbf{S}_{f.g.}$ containing infinite ceers, or
equivalently of the $\mathrm{c}$-degrees of ceers realized by word problems
of infinite finitely generated~semigroups.

We will work over the free semigroup on two generators: hence, let us
consider the alphabet $X:=\{a,b\}$ and denote by $X^+$ the set of nonempty
finite words of elements of $X$, with $\lambda$ denoting the empty word, and the
binary operation on $X^+$ given by the concatenation of words. It is well
known that $X^+$ together with this binary operation is the free semigroup on
two generators. In the rest of this discussion, $X^+$ will be identified by
coding with the set of natural numbers $\omega$.

\subsection{Finitely generated semigroups with dark word problem}

First, we remark that, for the case of semigroups, Theorem
\ref{thm:main-first} is optimal, in the sense that, while the word problem of
any finitely generated semigroup cannot be $\mathrm{c}$-equivalent to a
hyperdark ceer, there exist two-generator semigroups with dark word problem.
The following example is taken from Hirschfeldt and
Khoussainov~\cite{Hirschfeldt-Khoussainov} (see, in particular, Lemmas 2.2,
2.3, 2.4 and Theorem~3.7).

\begin{example}\label{example:Hirschfeldt-Khoussainov}
\cite{Hirschfeldt-Khoussainov} Let us first introduce some terminology and
notation. Given words $x,y \in X^+$, we say that $y$ is a \emph{subword} of
$x$ if there are words $u_1,u_2 \in X^+ \cup \{\lambda\}$ such that $x = u_1
y u_2$. Similarly, $y \in X^+$ is a \emph{subword} of an infinite sequence $f
\in X^{\omega}$ if there exist a word $u \in X^+ \cup \{\lambda\}$ and an
infinite sequence $g \in X^{\omega}$ such that $f = uyg$. We say that a word
or an infinite sequence $\alpha$ \emph{avoids} a finite word $y$ whenever
$\alpha$ does not contain $y$ as a subword. Finally, for every $l \in
\omega$, let us denote by $X^{\le l}$ and $X^{\ge l}$ the set of words on the
alphabet $X$ of length, respectively, at most $l$ and at least $l$.

For $Z \subseteq X^+$, let
$$
\overline{Z} := \{ u \in X^+: \ (\exists z \in Z)(\exists u_1, u_2 \in X^+
\cup \{ \lambda \}) [u = u_1 z u_2]   \}.
$$
In other words, $\overline{Z}$ is the set of words in $X^+$ containing a word
of $Z$ as a subword. Notice that $Z \subseteq \overline{Z}$. It is easy to
see that the unidimensional ceer $R_{\overline{Z}}$ is a congruence of $X^+$
for every set $Z \subseteq X^+$: hence, for any c.e.~set $Z \subseteq X^+$,
we get a finitely generated c.e.~semigroup $S_Z = X^+/_{R_{ \overline{Z}}}$.

By the discussion following Definition~\ref{def:various-darkness} we know
that a unidimensional ceer $R_X$ is dark if and only if $X$ is simple.
Therefore, in order to get that $S_Z$ is infinite with a dark word problem,
it suffices to build a simple set $Z \subseteq X^+$ such that $\overline{Z}$
is coinfinite: this ensures that $\overline{Z}$ is simple, being a coifinite
c.e.~superset of a simple set. This can be achieved using the following
result by Miller~\cite[Corollary 2.2]{Miller}: If a set $Y \subseteq X^+$
contains, for each $i$, at most one word of length $i+5$ and no words of
length $\leq 4$, then there is an infinite sequence $f \in X^{\omega}$ such
that $f$ avoids all the words in $Y$.

So suppose that $Z\subseteq X^+$ is a set such that $Z \cap X^{\le k+4}$
contains at most $k$ elements for every number $k$. From this and using the
fact that avoiding a word implies avoiding all its extensions, it is easy to
build a set $Y\subseteq X^+$, such that $Y$ contains exactly a string of
length $i+5$ for every number $i$, it contains no string of length $\leq 4$,
and if a string avoids all words in $Y$ then it avoids all words in $Z$ as
well. By Miller's result there is an infinite sequence $f\in X^\omega$
avoiding $Y$, and thus there are infinitely many finite words avoiding $Z$,
implying that $\overline{Z}$ is coinfinite.

Therefore, to complete our example it remains only to show that there exists
a simple set $Z \subseteq X^+$ such that for every $k$, $Z$ contains at most
$k$ words of length $\leq k+4$. This follows along the lines of the standard Post's
construction of a simple set (see, e.g.~\cite{Soare-Book-new}, Theorem 5.2.3):
Given a standard numbering $\{W_i: i \in \omega\}$ of the c.e. subsets of $X^+$, we let
$$Z = \{ u \in X^+: (\exists i)(\exists s)[u \in W_{i,s+1} \cap X^{\ge i+5} \ \text{and} \ W_{i,s} \cap X^{\ge i+5} = \emptyset] \}.$$
In other words, we enumerate
each set $W_i$ until a word $u$ of length at least $i+5$ appears: whenever this happens,
we enumerate $u$ into $Z$, and we do not put any more elements from $W_i$ into $Z$.
Notice that $Z$ is simple as every infinite c.e.~set $W_i$ must contain words of arbitrary length.
Moreover, by definition of $Z$, a word $u \in Z \cap X^{\le k+4}$ must have been taken from a set
$W_i$ with $i \le k-1$, which ensures that $Z$ contains at most $k$ such words.

\end{example}

As already anticipated in the introduction, we also notice that the above
example has been strikingly strengthened by Myasnikov and Osin in
\cite{Myasnikov-Osin}, where it is even built a finitely generated c.e.~group
with dark word problem.

\subsection{Finitely generated semigroups with non-dark word problem}

We partition $X^+$, the set of nonempty words on the alphabet $\{a,b \}$, as
follows:
\begin{enumerate}
\item $C:=\{ab^ia: i \ne 0\}$ (where for any string $u$, we denote by $u^i$
    the string obtained by concatenating $i$ times $u$ with itself). We
    refer to the elements of $C$ as \emph{coding words}.
\item $C_+$ consists of the words in $X^+$ which properly contain coding
    words as subwords, i.e.
    \[
C_+:=\{w \in X^+\smallsetminus C:  (\QE{v,v' \in X^+ \cup \{\lambda\}})
(\QE{u \in C})[w=vuv']\}.
\]

\item $C_{-}:=X^{+}\smallsetminus (C\cup C_+)$.
\end{enumerate}
Observe that the sets $C_{-}$, $C$, and $C_+$ are computable, infinite, and
they partition $X^+$.

\smallskip

Next, given  any ceer $R$, let $S(R)$ be the two-generator semigroup
presented by
\[
S(R):=\langle\, X \mid \{ab^{i+1}a\, =_{S(R)}\, ab^{j+1}a :
i \mathrel{R} j\ \}\cup \{v \,=_{S(R)}\, w : v,w \in C_{+}\}\,\rangle.
\] Thus, the
$=_{S(R)}$-closure of the set $C$ is partitioned in classes, with
representatives $ab^ia$ for each $i \ne 0$, and $ab^ia =_{S(R)} ab^ja$ if and
only if $(i-1) \mathrel{R} (j-1)$. The $={_{S(R)}}$-closure
$[C_+]_{={_{S(R)}}}$ of $C_+$ consists of just an equivalence class, say the
$=_{S(R)}$-equivalence class of $aaba$. Finally the set $[C_-]_{={_{S(R)}}}$
consists of an infinite bunch of singletons (namely the singletons of words
which avoid the words of the form $ab^ia$ with $i\neq 0$).

Recall that the \emph{uniform join} $U\oplus V$ of two equivalence relations
$U,V$ on $\omega$ is the equivalence relation on $\omega$ defined as $U\oplus
V:=\{(2x,2y): x\mathrel{U} y\} \cup \{(2x+1,2y+1): x\mathrel{V} y\}$.

\begin{lemma}\label{lem:R+Id realizes f.g. semi}
$=_{S(R)}$ is $\mathrm{c}$-equivalent (in fact, isomorphic, as defined in the
paragraph following Definition~\ref{def:reducibility}) with $R \oplus
\Id_\omega$, where we recall that $\Id_\omega$ denotes the equality relation
on $\omega$.
\end{lemma}

\begin{proof}
We recall the notion of the \emph{restriction} of a ceer $R$ to a nonempty
c.e.\ set $W$, see for instance \cite{AS-initial}: fix a computable surjection
$\pi: \omega \rightarrow W$, and define $R\!\restriction W$ to be the ceer
\[
x \rel{R\!\restriction W} y \Leftrightarrow \pi(x) \rel{R} \pi(y).
\]
It is immediate to see that, up to $\equiv_{\mathrm{c}}$, $R\!\restriction W$
does not depend on the chosen computable surjection.

By pairwise disjointness of the $=_{S(R)}$-closures of the computable sets
$C$, $C_+$, $C_-$ which yield a partition of $\omega$, we have that
\[
=_{S(R)} \equiv_{\mathrm{c}} \big( (=_{S(R)}\!\restriction C) \oplus
(=_{S(R)}\!\restriction C_+) \oplus
(=_{S(R)}\!\restriction C_-) \big).
\]
On the other hand, it is easily seen that $=_{S(R)}\!\restriction C
\equiv_{\mathrm{c}} R$, $=_{S(R)}\!\restriction C_+ \equiv_{\mathrm{c}}
\Id_1$ (where $\Id_1$ is the ceer with just one equivalence class) and
$=_{S(R)} \!\restriction C_- \equiv_{\mathrm{c}} \Id_\omega$. Hence $=_{S(R)}
\equiv_{\mathrm{c}} R \oplus \Id_1 \oplus \Id_\omega$, and from $\Id_\omega
\equiv_{\mathrm{c}} \Id_1 \oplus \Id_\omega$ we conclude
\[
=_{S(R)} \equiv_{\mathrm{c}} R\oplus \Id_\omega.
\]
Finally it is easy to see that all the bi-equivalences $\equiv_{\mathrm{c}}$
mentioned in the proof are in fact isomorphisms of equivalence relations, so
$=_{S(R)}$ is isomorphic with $R\oplus \Id_\omega$.
\end{proof}

\begin{theorem}\label{thm:R+Id realizes f.g. semi}
If $R$ is a ceer such that $R \equiv_{\mathrm{c}} R \oplus \Id_\omega$, then
there is a two-generator semigroup $S$ such that $=_{S} \equiv_{\mathrm{c}}
R$.
\end{theorem}

\begin{proof}
Given $R$, take $S:=S(R)$.
\end{proof}

\begin{remark}
Notice that $=_{S(R)}\in \Inf \smallsetminus \mathbf{dark}$. The ceers in
$\Inf \smallsetminus \mathbf{dark}$ are frequently called \emph{light}, and
their class is denoted as $\mathbf{light}$, see e.g.~\cite{joinmeet}.
\end{remark}

\subsection{How large is $\mathbf{S}_{f.g.}$?}

We conclude this section with a few remarks which suggests that the class of
ceers which are $\mathrm{c}$-realized by finitely generated semigroups is
unexpectedly large.

Theorem~\ref{thm:R+Id realizes f.g. semi} allows indeed to show that
$\mathbf{S}_{f.g.}^\infty$ (in fact the subclass of
$\mathbf{S}_{f.g.}^\infty$ consisting of the $\mathrm{c}$-degrees of light
ceers) embeds rich structures. First, recall the reducibility notion of the
following definition, where, for every $n\ge 1$, we fix a ceer $\Id_n$ with
exactly $n$ equivalence classes.

\begin{emdef}\cite{joinmeet}
For ceers $R$ and $S$,  $R \leq_{\mathcal{I}} S$ if and only if
$R\leq_{\mathrm{c}} S \oplus \Id_n$, for some number $n \ge 1$.
\end{emdef}

It is known (see \cite{joinmeet,andrews2020theory}) that the structure of the
$\mathcal{I}$-degrees of dark ceers, denoted as $\mathbf{dark}_{/
\mathcal{I}}$, is in a reasonable sense as complicated as possible: its first
order theory is computably isomorphic with first order arithmetic.

\begin{corollary}\label{cor:emb1}
$\mathbf{dark}_{/ \mathcal{I}}$ embeds into $\mathbf{S}_{f.g.}^\infty$.
\end{corollary}

\begin{proof}
Let $R$ and $S$ be any pair of dark ceers. Andrews and Sorbi \cite[Theorem
6.2]{joinmeet} proved that $R \! \equiv_{\mathcal{I}} \! S$ if and only if
$R\oplus \Id_\omega \equiv_{\mathrm{c}} S\oplus \Id_\omega$. This implies
that the map
\[
\iota:R \mapsto R\oplus \Id_\omega
\]
induces an embedding of $\mathbf{dark}_{/ \mathcal{I}}$ into
$\mathbf{light}_{/ \mathcal{I}}$, see \cite[Lemma 6.2]{andrews2020theory},
where of course $\mathbf{light}_{/ \mathcal{I}}$ denotes the
$\mathcal{I}$-degrees of light ceers. To conclude, it suffices to note that,
by Theorem~\ref{thm:R+Id realizes f.g. semi}, the image of the embedding
induced by $\iota$ is contained in $\mathbf{S}_{f.g.}^\infty$.
\end{proof}

Finally, we recall that (as proved in \cite{AS-initial}) one can embed the
tree $\left(\omega^{<\omega}, \subseteq \right)$ of finite strings of natural
numbers, partially ordered by the prefix relation on strings, as an initial
segment of $\Inf_{\mathrm{c}}$.

\begin{corollary}\label{cor:realized}
The tree $\left( \omega^{<\omega}, \subseteq \right)$ embeds as an initial
segment of $\Inf_{\mathrm{c}}$ in such a way that the range of the embedding
is included in $\mathbf{S}_{f.g.}^\infty$.
\end{corollary}

\begin{proof}
The proof follows by Theorem~\ref{thm:R+Id realizes f.g. semi} and by the
fact that \cite[Corollary~3.1]{AS-initial} shows that the tree
$\omega^{<\omega}$ can be embedded as an initial segment of
$\Inf_{\mathrm{c}}$ in such a way that the range of the embedding is included
in the $\mathrm{c}$-degrees of the ceers $R$ such that $R\equiv_{\mathrm{c}}
R\oplus \Id_\omega$.
\end{proof}

Notice that the embeddings mentioned in the proofs of both
Corollary~\ref{cor:emb1} and Corollary~\ref{cor:realized} take as images
$\mathrm{c}$-degrees of light ceers, in particular of ceers of the form $R
\oplus \Id_\omega$.

\subsection{What about finitely presented semigroups?}
It might be worth noticing that if $R \in \FinCl$ and $R$ is undecidable then
$S(R)$, as described above, is not finitely presentable. Indeed, Litvinceva
\cite{Litvinceva} showed that, if $S$ is a finitely presented semigroup such
that $=_S$ has only finitely many infinite equivalence classes, then $=_S$ is
decidable. Now, if $R \in \FinCl$ and $R$ is undecidable then $=_{S(R)}$ has
only one infinite equivalence class, namely $[C_+]_{=_{S(R)}}$, but on the
other hand $=_{S(R)}$ is undecidable, as $R \leq_{\mathrm{c}} =_{S(R)}$.

If one looks only for finitely generated semigroups $S$ such that $=_S \,
\in\FinCl$, then a slight modification of the proof of Lemma~\ref{lem:R+Id
realizes f.g. semi} shows how to build a finitely generated semigroup $S$,
such that $=_S \, \in \FinCl$ and $R\leq_{\mathrm{c}} =_S$, starting from any
ceer $R \in \FinCl$. For this, it is enough to take $S:=\langle\! X \mid
\{ab^{i+1}a\! =_{S}\! ab^{j+1}a : iRj \} \rangle$. To show that $=_S \, \in
\FinCl$, one easily sees that if $v \in C_+$ and $u_1, \ldots, u_n$ are the
subwords of $v$ of the form $ab^{i_j+1}a$ for some number $i_j$ with $1\leq j
\leq n$, occurring in distinct places of $v$, and $k_{j}$ is the number of
elements in the $R$-class of $i_j$, then the cardinality of the $=_S$-class
of $v$ equals the product $k_1 \cdot k_2\cdot \ldots \cdot k_n$. Once again,
by \cite{Litvinceva} if $R$ is undecidable then no such $S$ is finitely
presentable.

\section{Future research}
The research program of locating which $\mathrm{c}$-degrees of ceers are
$\mathrm{c}$-realized by word problems of suitable algebras is vast and, to
our knowledge, plenty of questions remain, untouched. We conclude by listing
a few research lines that may inspire future work.
\begin{enumerate}
\item Let $\mathbf{S}_{f.p.}$ and $\mathbf{S}_{c.e.}$ denote the
    $\mathrm{c}$-degrees of ceers which are $\mathrm{c}$-realized by the
    word problems of, respectively, finitely presented semigroups and
    c.e.~semigroups. As aforementioned in the introduction,
    $\mathbf{S}_{c.e.}$ coincides with $\mathbf{Ceers}_{\mathrm{c}}$ and,
    by Theorem \ref{thm:main-first} above, $\mathbf{S}_{c.e.}\smallsetminus
    \mathbf{S}_{f.g.}$ contains every hyperdark $\mathrm{c}$-degree. Then,
    it is natural to ask if $\mathbf{S}_{f.g.}\smallsetminus
    \mathbf{S}_{f.p.}$ is also nonempty. If this is the case, one may try
    to compare $\mathbf{S}_{f.g.}$ and $\mathbf{S}_{f.p.}$  with respect to
    the initial segments of $\Inf_{\mathrm{c}}$ that they realize in the
    sense of Corollary~\ref{cor:realized}.

\item More generally, it may be interesting to move the focus from
    semigroups to other algebraic varieties $\mathcal{V}$, and to
    investigate which $\mathrm{c}$-degrees of ceers are
    $\mathrm{c}$-realized by members of $\mathcal{V}$, and in particular by
    finitely presented or finitely generated members of $\mathcal{V}$. The
    case of groups appears to be natural and challenging at the same time.
\item The literature is also rich of papers studying presentations of
    structures but using coceers instead of ceers (a \emph{coceer} is an
    equivalence relation whose complement is c.e.), and investigating which
    structures have a ``word problem'' coinciding with a given coceer: see
    for instance \cite{Khoussainov-Slaman-Semukhin, Kasymov-Morozov,
    Godziszewski}. A more recent variation (\cite{Khoussainov-quest-I,
    H.Trainor-Khoussainov-Turetsky}) studies which random structures have a
    ``word problem'' coinciding with a given ceer or coceer. In all these
    cases, it would be interesting to generalize these approaches to
    $\mathrm{c}$-realizability, aiming at describing those structures
    having word problems which are $\mathrm{c}$-realizable (not just merely
    coinciding) with given ceers or coceers.
\end{enumerate}


\end{document}